\documentclass{amsart}
\usepackage{amsmath, amscd, amssymb, amsthm}
\usepackage{bbm}
\usepackage{latexsym}
\usepackage{amsfonts}
\usepackage{graphicx}
\usepackage[all,cmtip]{xy}
\usepackage[colorlinks,linkcolor=blue,breaklinks=blue,urlcolor=blue,citecolor=blue,anchorcolor=blue,pagebackref]{hyperref}%
\usepackage{geometry}
\newtheorem{lemma}{Lemma}

\newtheorem{proposition}{Proposition}

\renewcommand*\backref[1]{}
\renewcommand*\backrefalt[4]{ \ifcase #1 \or (cited on page #2) \else (cited on pages #2) \fi}%(no citations)

\newcommand{\be}{\begin{equation}}
\newcommand{\ee}{\end{equation}}
\newcommand{\bea}{\begin{eqnarray}}
\newcommand{\eea}{\end{eqnarray}}

\newcommand{\vs}{\vspace{0.5cm}}

\def\XXint#1#2#3{{\setbox0=\hbox{$#1{#2#3}{\int}$ }
\vcenter{\hbox{$#2#3$ }}\kern-.6\wd0}}

\begin{document}

\title[Streets-Tian Conjecture on Lie algebras]{Streets-Tian Conjecture on Lie algebras with codimension $2$ abelian ideals}

\author{Kexiang Cao}
\address{Kexiang Cao. School of Mathematical Sciences, Chongqing Normal University, Chongqing 401331, China}
\email{{ caokx1214@qq.com}}\thanks{Cao is supported by Chongqing graduate student research grant No.\,CYB240231. The corresponding author Zheng is partially supported by NSFC grants 12141101 and 12471039, by Chongqing grant cstc2021ycjh-bgzxm0139, by Chongqing Normal University grant 19XRC001, and is supported by the 111 Project D21024.}

\author{Fangyang Zheng}
\address{Fangyang Zheng. School of Mathematical Sciences, Chongqing Normal University, Chongqing 401331, China}
\email{20190045@cqnu.edu.cn; franciszheng@yahoo.com} \thanks{}


\begin{abstract}
A Hermitian-symplectic metric is a Hermitian metric whose K\"ahler form is given by the $(1,1)$-part of a closed $2$-form. Streets-Tian Conjecture states that a compact complex manifold admitting a Hermitian-symplectic metric must be K\"ahlerian (i.e., admitting a K\"ahler metric). The conjecture is known to be true in dimension $2$ but is open in dimensions $3$ or higher in general, except in a number of special situations, such as twistor spaces (Verbitsky), Fujiki ${\mathcal C}$ spaces (Chiose), Vaisman manifolds (Angella-Otiman), etc.  For Lie-complex manifolds (namely, compact quotients $G/\Gamma$ of  Lie groups by  discrete subgroups with left-invariant complex structures), the conjecture has also been confirmed in a number of special cases, including when $G$ is nilpotent (Enrietti-Fino-Vezzoni), when $G$ is completely solvable (Fino-Kasuya), or when $J$ is abelian (Fino-Kasuya-Vezzoni), or $G$ is almost abelian  (Fino-Kasuya-Vezzoni, Fino-Paradiso), etc. In this article, we conduct a detailed case analysis and confirm Streets-Tian Conjecture for $G$ whose Lie algebra contains an abelian ideal of codimension $2$. Such Lie algebras are always solvable of step at most $3$, but are not $2$-step solvable and not completely solvable in general. Our approach is explicit in nature by describing both the Hermitian-symplectic metrics on such Lie algebras and the pathways of deforming them into K\"ahler ones, in hope of advancing our understanding of the subtlety and intricacy of this interesting conjecture in non-K\"ahler geometry. 
\end{abstract}

\maketitle

\tableofcontents

\section{Introduction and statement of results}\label{intro}

An interesting problem in Hermitian geometry is the following conjecture of  Streets and Tian (\cite{ST}), which states that any compact complex manifold admitting a Hermitian-symplectic metric must be K\"ahlerian, i.e., it admits a K\"ahler metric. Streets and Tian introduced the notion of {\em Hermitian-symplectic metrics,} which means a Hermitian metric $g$ whose K\"ahler form $\omega$ is the $(1,1)$-part of a closed $2$-form. That is, there exists a global $(2,0)$-form $\alpha$ on the manifold so that $\Omega = \alpha + \omega + \overline{\alpha}$ is closed. Equivalently, this means a compact complex manifold $M^n$ admitting a symplectic form (i.e., non-degenerate closed $2$-form) $\Omega$ so that $\Omega (x, Jx)>0$ for any non-zero tangent vector $x$. For this reason, a Hermitian-symplectic structure is also called a {\em symplectic structure taming a complex structure $J$} (\cite{EFV}).

Hermitian-symplectic metrics are a natural mix of Hermitian and symplectic structures, two classic objects of study in geometry. By definition, it is clear  that any Hermitian-symplectic metric would satisfy the condition $\partial \overline{\partial} \omega =0$, namely, it is {\em pluriclosed,}  a type of special Hermitian metrics that has been extensively studied (see for instance the excellent survey paper \cite{FinoTomassini} and the references therein). 
In a series of papers, Streets and Tian developed the important theory of Hermitian curvature flow (\cite{ST11}, see also \cite{Streets,  ST12, ST13}), and Hermitian-symplectic metrics play a big role in the case of pluriclosed flow, which is a special case of Hermitian curvature flow. Recently, Ye \cite{Ye} proved that the Hermitian-symplectic  property is preserved under the pluriclosed flow. Hermitian-symplectic metrics also enjoy some nice properties, for instance, it is stable under small deformations (\cite[Ch.12]{OV}), and the existence of a Hermitian-symplectic metric would imply the existence of a {\em strongly Gauduchon} metric \cite[Lemma 1]{YZZ}, which means (see \cite{Popovici}) a Hermitian metric whose K\"ahler form $\omega$ satisfies $\partial \omega^{n-1} = \overline{\partial} \Phi$ for some $(n,n-2)$-form $\Phi$, where $n$ is the complex dimension.

The Streets-Tian Conjecture is known to be true in complex dimension $2$ (\cite{LiZhang, ST}). See also \cite{Donaldson} for a related more general conjecture in real dimension $4$. In dimensions $3$ or higher, the conjecture is open in general but known to be true in several special cases. For instance, Verbitsky \cite{Verbitsky} showed that any non-K\"ahlerian twistor space does not admit any pluriclosed metric; Chiose \cite{Chiose} proved that any Fujiki ${\mathcal C}$ class manifold (i.e., compact complex manifolds bimeromorphic to compact K\"ahler manifolds) does not admit any pluriclosed metric; Fu, Li and Yau in \cite{FuLiYau} proved that an important  special type of non-K\"ahler Calabi-Yau threefolds does not admit any pluriclosed metric; Di Scala-Lauret-Vezzoni in \cite{DLV} proved that compact Chern flat manifolds do not admit any pluriclosed metrics. Since Hermitian-symplectic metrics are always pluriclosed,  we know that {\em Streets-Tian Conjecture holds for all twistor spaces, all Fujiki ${\mathcal C}$ class manifolds, all the special non-K\"ahler Calabi-Yau threefolds of Fu-Li-Yau, and all compact Chern flat manifolds.}

Recently, Angella and Otiman \cite{AOt} give a systematic study of special Hermitian metrics on Vaisman manifold. In particular,  they showed that any (non-K\"ahler) Vaisman manifold does not admit any Hermitian-symplectic metric (as well as a number of other interesting special types of Hermitian metrics). So {\em Streets-Tian Conjecture holds for all Vaisman manifolds.}  
Recall that  Vaisman manifolds are a special type of  locally conformally K\"ahler manifolds. More precisely, a Vaisman manifold is a Hermitian manifold $(M^n,g)$ whose K\"ahler form $\omega$ satisfies $d\omega = \psi \wedge \omega$ for a closed $1$-form $\psi$, and $\psi$ is parallel under the Levi-Civita connection. The book \cite{OV} forms an encyclopedia for locally conformally K\"ahler and Vaisman manifolds.

In a recent work \cite{GuoZ2},  Guo and the second named author confirmed Streets-Tian Conjecture for all compact non-balanced Bismut torsion parallel (BTP) manifolds. A Hermitian metric is BTP if its Bismut connection has parallel torsion (\cite{ZhaoZ22, ZhaoZ24}), and it is balanced if Gauduchon's torsion $1$-form (\cite{Gau84}) vanishes, or equivalently, if $d(\omega^{n-1})=0$ where $\omega$ is the K\"ahler form and $n$ is the complex dimension. By the result of Andrada and Villacampa \cite{AndV}, Vasiman manifolds are always (non-balanced) BTP. The class of non-balanced BTP manifolds also include all Bismut K\"ahler-like (BKL) manifolds as a subset, the latter means that the curvature of the Bismut connection obeys all K\"ahler symmetries. See \cite{YZZ, ZZ-Crelle, ZZ-JGP} for more discussions on such metrics. \cite{GuoZ2} also showed that Streets-Tian Conjecture holds for all  Chern K\"ahler-like manifolds, meaning the curvature of Chern connection obeys all K\"ahler symmetries. This includes  Chern flat manifolds as a subset.

Next let us focus on  {\em Lie-complex manifolds}: compact quotients $M=G/\Gamma$ of Lie groups by discrete subgroups, with the complex structure (when lifted onto $G$) being left-invariant. When $G$ is nilpotent (solvable), we shall call $M$ a complex {\em nilmanifold} ({\em solvmanifold}).  An important supporting evidence to Streets-Tian Conjecture is the theorem of Enrietti, Fino and Vezzoni \cite{EFV} which states that any complex nilmanifold cannot admit a Hermitian-symplectic metric, unless it is a complex torus. So {\em Streets-Tian Conjecture holds for all complex nilmanifolds} (see also \cite{AN} for an important supplement). Nilmanifolds form a large class of special Hermitian manifolds, with rich topological varieties and are often used as a testing ground in non-K\"ahler geometry. 

Fino and Kasuya confirmed Streets-Tian Conjecture \cite{FK} for all completely solvable groups (meaning that $\mbox{ad}_x$ has only real eigenvalues for any $x$ in the corresponding Lie algebra). In \cite{FKV}, Fino-Kasuya-Vezzoni also confirmed the conjecture in a number of cases, including all Oeljeklaus-Toma manifolds, and all Lie-complex manifolds with either abelian complex structure $J$ or with $J$-invariant nilpotent complements, and all almost abelian manifolds that are either in real dimension $6$ or not of type (I), meaning that there is $x\in {\mathfrak g}$ such that $\mbox{ad}_x$ has an eigenvalue with non-zero real part. In \cite{FP4}, Fino and Paradiso confirmed the conjecture for all almost abelian Lie algebras. 

In \cite{GuoZ2}, Guo and the second named author also confirmed Streets-Tian Conjecture for Lie-complex manifolds $G/\Gamma$ where the Lie algebra ${\mathfrak g}$ of $G$ contains an abelian ideal ${\mathfrak a}$ of codimension $2$, with $J{\mathfrak a} = {\mathfrak a}$. This leaves us the with the natural wondering of what happens when $J{\mathfrak a} \neq {\mathfrak a}$? The main purpose of this article is to confirm Streets-Tian Conjecture  in this case:

\begin{proposition} \label{thm1}
Let $M^n=G/\Gamma $ be a compact complex manifold which is the quotient of a Lie group $G$ by a discrete subgroup $\Gamma$, where the complex structure $J$ of $M$ (when lifted onto $G$) is left-invariant. Assume that the Lie algebra ${\mathfrak g}$ of $G$ contains an abelian ideal ${\mathfrak a}$  of codimension $2$. Assume that $J{\mathfrak a} \neq {\mathfrak a}$. If $M$ admits a Hermitian-symplectic metric, then it admits a (left-invariant) K\"ahler metric. 
\end{proposition}

In other words, by combining the above with \cite[Proposition 2]{GuoZ2}, we know that Streets-Tian Conjecture holds for all Lie-complex manifolds when the Lie algebra contains an abelian ideal of codimension $2$, a natural generalization to the almost abelian case. The latter was actively studied in recent years by many, see for instance \cite{AO, AL, Bock, CM, FP1, FP2, FP3, GuoZ, LW, Paradiso, Podesta, Ugarte}.

Note that when a Lie algebra ${\mathfrak g}$ contains an abelian ideal of codimension $2$, it will always be solvable of step at most $3$, but in general it won't be $2$-step solvable or completely solvable. In particular the above statement is not covered by the aforementioned work of Fino-Kasuya, Fino-Kasuya-Vezzoni, or Fino-Paradiso on special solvmanifolds.

Unlike almost abelian Lie algebras, those  with codimension $2$ abelian ideals present adequate algebraic complexity when Hermitian structures are involved. In \cite{CaoZ}, we were able to verify Fino-Vezzoni Conjecture on such Lie algebras, but we only had partial description of balanced or pluriclosed metrics there. Fino-Vezzoni Conjecture states that a compact complex manifold admitting a balanced metric and a pluriclosed metric must be K\"ahlerian. It is closely related to Streets-Tian Conjecture, but these two conjectures do not imply each other. 

Suppose that a Lie algebra ${\mathfrak g}$ contains an abelian ideal ${\mathfrak a}$ of codimension $2$. It can be divided into two cases: when ${\mathfrak g}/{\mathfrak a}$ is non-abelian or when it is abelian. Now let us consider all the complex structures $J$ on ${\mathfrak g}$ such that  $J{\mathfrak a}\neq {\mathfrak a}$. In the first case, namely when ${\mathfrak g}/{\mathfrak a}$ is non-abelian, they  can be further divided into three subcases: generic, half-generic, or degenerate. In the second case, namely when ${\mathfrak g}/{\mathfrak a}$ is abelian, they can be divided into three subcases, depending on whether $r_0=0$, $r_0=1$, or $r_0=2$. Here $r_0$ denotes the dimension of $({\mathfrak a}_J + [{\mathfrak g}, {\mathfrak g}])/{\mathfrak a}_J$, where $ {\mathfrak a}_J = {\mathfrak a} \cap J{\mathfrak a}$. 

In each subcase the Hermitian geometric behavior is somewhat different, and the existence of Hermitian-symplectic metrics translates into a system of matrix equations. Through case analysis, one can write down explicitly those $J$ which admits Hermitian-symplectic metrics, and find out explicit pathways to deform them into K\"ahler ones. In this sense our argument is elementary and explicit, not relying on the powerful symplectic Lie algebra theory but rather focusing more on the Hermitian geometric aspects. We hope these detailed analysis could be useful in the study of other Hermitian geometric problems on such manifolds. 

Lie algebras with codimension $2$ abelian ideals form a rich yet special class of Hermitian manifolds which already present  adequate algebraic complexity. Verifying Streets-Tian Conjecture in an explicit manner for this particular class could therefore  illustrate both the technical difficulty as well as the subtlety and intricacy of the conjecture.

\vspace{0.3cm}

\section{Lie-Hermitian manifolds and Hermitian-symplectic metrics}

In this section we will recall the basics of Hermitian Lie algebras and the characterization of Hermitian-symplectic metrics on them. It will be parallel to Section 3 of \cite{GuoZ2}, but we will repeat some of the discussions here for the convenience of the readers and also to make the presentation self-contained.

First let us recall some basics for Lie-Hermitian manifolds. We refer the readers to \cite{VYZ, GuoZ, CaoZ} for more details. By Lie-Hermitian manifolds we mean compact Hermitian manifolds $(M^n,g)$ in the form $M=G/\Gamma$, where $G$ is an even-dimensional Lie group, $\Gamma$ is a discrete subgroup,  $J$ is a left-invariant complex structure on $G$, and $g$ is a left-invariant Riemannian metric on $G$ compatible with $J$. 

Note that the compactness of $M$ forces $G$ to be unimodular, and $G$ admits a bi-invariant measure \cite{Milnor}. Using this measure, Fino and Grantcharov \cite{FG} proved that, given any pluriclosed metric $g_0$ on $M$, by averaging it over the manifold with respect to the aforementioned measure, one can produce another pluriclosed metric $g$ which is now left-invariant. Similarly, Enrietti, Fino and Vezzoni proved in \cite{EFV} that if $M$ admits a Hermitian-symplectic metric, then through averaging one would get a left-invariant  Hermitian-symplectic metric. (In this regard, the averaging trick for balanced metric was due to Ugarte in \cite{Ugarte}). Therefore, in order to verify Streets-Tian Conjecture for Lie-Hermitian manifolds, it suffices to consider Hermitian-symplectic metrics that are left-invariant. 

Denote by ${\mathfrak g}$  the Lie algebra of $G$. As is well-known, left-invariant complex structures on $G$ are in one-one correspondence to complex structures on ${\mathfrak g}$, which means a linear transformation $J$ on ${\mathfrak g}$ satisfying $J^2=-I$ and the integrability condition
\begin{equation} \label{integrability}
[x,y] - [Jx,Jy] + J[Jx,y] + J[x,Jy] =0, \ \ \ \ \ \ \forall \ x,y \in {\mathfrak g}. 
\end{equation}
Likewise, left-invariant metrics on $G$ correspond to inner products on ${\mathfrak g}$, and compatibility with complex structure simply means the inner product will make $J$ orthogonal. 

Let us denote by ${\mathfrak g}^{\mathbb C}$ the complexification of ${\mathfrak g}$, and by ${\mathfrak g}^{1,0}= \{ x-\sqrt{-1}Jx \mid x \in {\mathfrak g}\} \subseteq {\mathfrak g}^{\mathbb C}$. Then (\ref{integrability}) is equivalent to the condition that ${\mathfrak g}^{1,0}$ is a complex Lie subalgebra of ${\mathfrak g}^{\mathbb C}$. From now on, we will extend the metric $g=\langle , \rangle $ bilinearly over ${\mathbb C}$, and a basis $e$ of the complex vector space ${\mathfrak g}^{1,0}$ will be called a {\em frame} for the Hermitian Lie algebra $( {\mathfrak g}, J, g)$. Let $e=\{ e_1, \ldots , e_n\}$ be a frame of ${\mathfrak g}$, and denote by $\{ \varphi_1, \ldots , \varphi_n\}$ the dual coframe, namely, a basis of the dual complex vector space $({\mathfrak g}^{1,0})^{\ast}$ such that $\varphi_i(e_j)=\delta_{ij}$, $\forall \ 1\leq i,j\leq n$. Following \cite{VYZ}, we will use
\begin{equation*} 
C^j_{ik} = \varphi_j( [e_i,e_k]), \ \ \ \ \ \  D^j_{ik} = \overline{\varphi}_i( [\overline{e}_j, e_k] )
\end{equation*}
to denote the structure constants, namely, under the frame $e$ we have
\begin{equation*} \label{CD}
[e_i,e_k] = \sum_j C^j_{ik}e_j, \ \ \ \ \ [e_i, \overline{e}_j] = \sum_k \big( \overline{D^i_{kj}} e_k - D^j_{ki} \overline{e}_k \big) .
\end{equation*}
Equivalently, one has the structure equation
\begin{equation} \label{structure}
d\varphi_i = -\frac{1}{2} \sum_{j,k} C^i_{jk} \,\varphi_j\wedge \varphi_k - \sum_{j,k} \overline{D^j_{ik}} \,\varphi_j \wedge \overline{\varphi}_k, \ \ \ \ \ \ \forall \ 1\leq i\leq n.
\end{equation}
Note that ${\mathfrak g}$ is unimodular if and only if $\mbox{tr}(ad_x)=0$ for any $x\in {\mathfrak g}$, which is equivalent to
\begin{equation} \label{unimodular}
{\mathfrak g} \ \, \mbox{is unimodular}  \ \ \Longleftrightarrow  \ \ \sum_r \big( C^r_{ri} + D^r_{ri}\big) =0 , \, \ \forall \ i.
\end{equation}
So far we did not assume that the frame $e$ is unitary, namely, the matrix  $g=(g_{i\bar{j}})=(\langle e_i, \overline{e}_j\rangle )$ is equal to the identity matrix. To get simplified formula, we will assume that $e$ is unitary from now on (without this assumption many of the formula will be clumsier and contain $g$ and its inverse here and there).  Denote by $\nabla$ the Chern connection and by $T$ its torsion tensor. We have
\begin{equation*} \label{Gamma}
\nabla e_i = \sum_j \theta_{ij} e_j, \ \ \ \ \theta_{ij} = \sum_k \big( \Gamma^j_{ik} \varphi_k - \overline{\Gamma^i_{jk}}\, \overline{\varphi}_k \big), \ \ \ \ \ \Gamma^j_{ik} = D^j_{ik},
\end{equation*}
where $\theta$ is the connection matrix of $\nabla$ under $e$. The torsion tensor $T$ of $\nabla$ has components
\begin{equation*} \label{torsion}
T( e_i, \overline{e}_j)=0, \ \ \ \ T(e_i,e_j)  = \sum_k T^k_{ij}e_k, \ \ \ \  \ \ T^j_{ik}= - C^j_{ik} -  D^j_{ik} + D^j_{ki}.
\end{equation*}
Differentiate the structure equation (\ref{structure}), we get the  first Bianchi identity, which is equivalent to the Jacobi identity in this case:
\begin{equation} \label{Jacobi}
\left\{ 
\begin{split} 
  \sum_r \big( C^r_{ij}C^{\ell}_{rk} + C^r_{jk}C^{\ell}_{ri} + C^r_{ki}C^{\ell}_{rj} \big) \ = \ 0,  \hspace{3.1cm}\\
  \sum_r \big( C^r_{ik}D^{\ell}_{jr} + D^r_{ji}D^{\ell}_{rk} - D^r_{jk}D^{\ell}_{ri} \big) \ = \ 0,  \hspace{2.9cm} \\
  \ \ \sum_r \big( C^r_{ik}\overline{D^r_{j \ell}}  - C^j_{rk}\overline{D^i_{r \ell}} + C^j_{ri}\overline{D^k_{r \ell}} -  D^{\ell}_{ri}\overline{D^k_{j r}} +  D^{\ell}_{rk}\overline{D^i_{jr}}  \big) \ = \ 0 . 
\end{split} 
\right. 
\end{equation}
For our later purpose of verifying Streets-Tian Conjecture for Lie algebras, we need the following characterization of Hermitian-symplectic metrics which is \cite[Lemma 3]{GuoZ2}, and we include its short proof here for readers convenience.
\begin{lemma}[\cite{GuoZ2}] \label{lemma1}
Let $({\mathfrak g}, J, g)$ be a Lie algebra equipped with a Hermitian structure. The metric $g$ is Hermitian-symplectic if and only if for any given unitary frame $e$, there exists a skew-symmetric matrix $S$ such that
\begin{equation} \label{eq:HS}
\left\{ \begin{split} 
 \sum_r \big( S_{ri}C^r_{jk} + S_{rj} C^r_{ki}  + S_{rk} C^r_{ij} \big) =0,  \\
 \sum_r \big( S_{rk}\overline{D^i_{rj}} - S_{ri} \overline{D^k_{rj}} \big) = - \sqrt{-1}T^j_{ik}, \,     \end{split} \right. \ \ \ \ \ \ \ \ \ \ \ \ \ \forall\ 1\leq i,j,k\leq n. 
 \end{equation}
\end{lemma}

\begin{proof}
Let $\varphi$ be the unitary coframe dual to $e$. By definition, $g$ will be Hermitian-symplectic if and only if there exists a $(2,0)$-form $\alpha$ so that $\Omega = \alpha + \omega + \overline{\alpha}$ is closed. This means that $\partial \alpha =0$ and $\overline{\partial}\alpha = - \partial \omega$. Write $\alpha = \sum_{i,k} S_{ik}\varphi_i\wedge \varphi_k$, where $S$ is a skew-symmetric matrix. Then we have
\begin{eqnarray*}
 \partial \alpha & = &  2\sum_{r,k} S_{rk} \partial \varphi_r \wedge \varphi_k \ = \ - \sum_{r,i,j,k} S_{rk} C^r_{ij} \, \varphi_i\wedge \varphi_j \wedge \varphi_k \\
 & = & -\frac{1}{3} \sum_{i,j,k} \{ \sum_r \big(  S_{rk} C^r_{ij} + S_{rj} C^r_{ki} + S_{ri} C^r_{jk} \big)\}  \,\varphi_i\wedge \varphi_j \wedge \varphi_k. 
\end{eqnarray*}
So by $\partial \alpha =0$ we get the first identity in (\ref{eq:HS}). Similarly, we have
\begin{eqnarray*}
 \overline{\partial} \alpha & = &  2\sum_{r,k} S_{rk} \overline{\partial}\varphi_r \wedge \varphi_k \ = \ - 2 \sum_{r,i,j,k} S_{rk} \overline{D^i_{rj}} \, \varphi_i\wedge \overline{\varphi}_j \wedge \varphi_k \\
 & = & - \sum_{i,j,k} \{ \sum_r \big( S_{rk} \overline{D^i_{rj}} - S_{ri} \overline{D^k_{rj}} \big)\}  \, \varphi_i\wedge \overline{\varphi}_j \wedge \varphi_k. 
\end{eqnarray*}
On the other hand,
$$ \partial \omega = \sqrt{-1} \,^t\!\tau \wedge \overline{\varphi} = \sqrt{-1} \sum_{i,j,k} T^j_{ik} \,\varphi_i\wedge \varphi_k \wedge \overline{\varphi}_j , $$
so by comparing the above two lines we get the second identity in (\ref{eq:HS}). This completes the proof of Lemma \ref{lemma1}.
\end{proof}

For our later proofs, we will need another lemma:

\begin{lemma} \label{lemma1b}
Let $(M^n,g)$ be a compact Hermitian manifold. If $g$ is Hermitian-symplectic, then any non-negative $(p,p)$-form $\Phi$ which is $d$-exact  must be trivial.
\end{lemma}

\begin{proof}
Let $\omega$ be the K\"ahler form of $g$. By assumption, there exists a $(2,0)$-form $\alpha$ so that the $2$-form $\Omega = \alpha + \omega + \overline{\alpha}$ is $d$-closed. Clearly, $\alpha \overline{\alpha}$ is non-negative, so is $(\alpha \overline{\alpha})^k$ for any $k$. The $(k,k)$-part of $\Omega^k$ is
$$ (\Omega^k)^{k,k} = \omega^k + C_k^2C_2^1\omega^{k-2} (\alpha \overline{\alpha}) + C_k^4C_4^2\omega^{k-4} (\alpha \overline{\alpha})^2 + \cdots \geq \omega^k. $$
Therefore by taking $k=n-p$ we get
$$ \int_M \Phi \wedge \Omega^{n-p} = \int_M \Phi \wedge (\Omega^{n-p})^{n-p,n-p} \geq \int_M \Phi \wedge \omega^{n-p} \geq 0. $$
The left hand side is zero since $\Phi$ is $d$-exact and $\Omega$ is $d$-closed, while the far right hand side is zero only if $\Phi$ is trivial since $\omega$ is a metric. This completes the proof of the lemma.
\end{proof}

\vspace{0.3cm}

\section{Lie algebras with codimension $2$ abelian ideals}

In this section, we will recall the properties for Lie algebras containing abelian ideals of codimension $2$, following the discussions and computations in our previous work \cite{CaoZ}. We will only collect the part that will be needed for our  proof of Proposition \ref{thm1} later, and we refer the interested  readers to \cite{CaoZ} for more details.

Throughout this section, we will assume that {\em  ${\mathfrak g}$ is a unimodular Lie algebra of real dimension $2n$ containing an abelian ideal ${\mathfrak a}$ of codimension $2$,  $J$  is a complex structure on ${\mathfrak g}$, and $g=\langle , \rangle $ is an inner product on ${\mathfrak g}$ compatible with $J$. We will also assume that $J{\mathfrak a} \neq {\mathfrak a}$.}

Note that we excluded the $J{\mathfrak a} = {\mathfrak a}$ case here because its algebraic behavior is quite different and the validity of Streets-Tian Conjecture for it has been established in \cite{GuoZ2}. We will divide the discussion into two cases, depending on whether the $2$-dimensional Lie algebra ${\mathfrak g}/{\mathfrak a}$ is abelian or not.

\vspace{0.2cm}

\noindent {\bf Case 1:} $J{\mathfrak a} \neq {\mathfrak a}$ and ${\mathfrak g}/{\mathfrak a}$ is not abelian.

\vspace{0.2cm}

Following \cite{CaoZ}, let us introduce the notations  
\begin{equation*} \label{a'andb}
  {\mathfrak a}_J:= {\mathfrak a} \cap J {\mathfrak a}, \ \ \ \ \ {\mathfrak a}':=  {\mathfrak a} + [{\mathfrak g}, {\mathfrak g}], \ \ \ \ \ {\mathfrak b} = {\mathfrak a} \cap J {\mathfrak a}'.
\end{equation*}
Under our assumption that $J{\mathfrak a}\neq  {\mathfrak a}$ and ${\mathfrak g}/{\mathfrak a}$ is non-abelian, we know that ${\mathfrak a}_J$  and ${\mathfrak a}'$ will respectively have codimension $4$ and $1$ in ${\mathfrak g}$. By \cite[Lemma 3]{CaoZ} we have
$${\mathfrak a}_J \subsetneq {\mathfrak b}  \subsetneq {\mathfrak a} \subsetneq {\mathfrak a}' \subsetneq {\mathfrak g} .$$
Furthermore, $x\in {\mathfrak b}\setminus {\mathfrak a}_J$ would imply $Jx \in {\mathfrak a}' \setminus {\mathfrak a}$ and $y\in {\mathfrak a}\setminus {\mathfrak b}$ would imply $Jy \notin {\mathfrak a}' $. The Lie algebra with complex structures $({\mathfrak g},J)$ (or for simplicity, we simply say $J$) will be divided into three mutually  disjoint subsets:
\begin{equation*}
J \ \mbox{is said to be}  \left\{ \begin{split}  \mbox{{\em generic}, \ \ \ \ \ \ if} \ [J{\mathfrak b}, {\mathfrak b}] \not\subseteq {\mathfrak b}; \hspace{2.9cm} \\
  \mbox{{\em half-generic}, \ \ \ if}  \ [J{\mathfrak b}, {\mathfrak b}] \subseteq {\mathfrak b} \ \mbox{and} \  [J{\mathfrak b}, {\mathfrak a}] \not\subseteq {\mathfrak a}_J; \\
 \mbox{{\em degenerate}, \ \ \ if} \  [J{\mathfrak b}, {\mathfrak a}] \subseteq {\mathfrak a}_J.  \hspace{2.55cm}
\end{split} \right. 
\end{equation*}
Now suppose $({\mathfrak g},J)$ is equipped with a metric $g$. Let $x\in {\mathfrak b}\cap  {\mathfrak a}_J^{\perp}$ and $y\in {\mathfrak a}\cap  {\mathfrak b}^{\perp}$ be unit vectors, they are uniquely determined up to signs. Write $\dim_{ \mathbb R}{\mathfrak g}=2n$. Then by \cite{CaoZ}, a  unitary basis $\{ e_1, \ldots , e_n\}$ of ${\mathfrak g}^{1,0}$ is called an {\em admissible frame} of ${\mathfrak g}$, if 
\begin{equation*} \label{def}
{\mathfrak a}_J = \mbox{span}\{ e_j+\overline{e}_j, \, i(e_j-\overline{e}_j); \ 3\leq j\leq n\}, \ \ \ \ {\mathfrak a} ={\mathfrak a}_J + \mbox{span}\{x,y\}, 
\end{equation*} 
and 
\begin{equation*} \label{def2}
e_1=\frac{1}{\sqrt{2}}(x-iJx), \ \ \ Y= \frac{1}{\sqrt{2}}(y-iJy) = i\delta e_1 + \delta'e_2,
\end{equation*} 
where $\delta = \langle Jx, y\rangle \in (-1,1)$ and $\delta '=\sqrt{1-\delta^2} \in (0,1]$. The following is \cite[Lemma 4]{CaoZ}:

\begin{lemma} [ \cite{CaoZ} ]   \label{lemma2}
Let $({\mathfrak g},J,g)$ be a unimodular Hermitian Lie algebra which contains an abelian ideal ${\mathfrak a}$ of codimension $2$ such that $J{\mathfrak a}\neq  {\mathfrak a}$ and ${\mathfrak g}/{\mathfrak a}$ is non-abelian. Let $e$ be an admissible frame. Then the structure constants $C$ and $D$ satisfy 
\begin{eqnarray*}
&& C^{\ast}_{ij} = D^{\ast}_{\ast i} = C^{\alpha }_{\beta i} = D^i_{\alpha\beta} =0,   \label{eqCD1}\\
&& C^{\ast}_{1i}= \overline{D^i_{\ast 1}} , \ \ \ C^{\ast}_{2i}= \overline{D^i_{\ast 2}} -2t \overline{D^i_{\ast 1}}, \ \ \ C^{\ast}_{12}= \overline{D^2_{\ast 1}} -  \overline{D^1_{\ast 2}} + 2t \overline{D^1_{\ast 1}},  \label{eqCD2}
\end{eqnarray*}
for any $1\leq \alpha , \beta \leq 2$ and any $3\leq i,j\leq n$, any $1\leq \ast \leq n$. Here $t=\frac{i\delta}{\delta'}$.
\end{lemma}

So under an admissible frame $e$ all the  $C$ components can be expressed in terms of $D$ components, and the only possibly non-zero $D$ components  are the following:
\begin{equation*} \label{11}
D_{\alpha} = (D^{\,\ast}_{\ast \alpha}) = \left[ \begin{array}{ll} E_{\alpha} & 0 \\ V_{\alpha} & Y_{\alpha}  \end{array} \right] , \ \ \ \ \ E_{\alpha} = \left[ \begin{array}{ll} D^{\,1}_{1\alpha} &  D^{\,2}_{1\alpha} \\  D^{\,1}_{2\alpha} &  D^{\,2}_{2\alpha}  \end{array} \right] , \ \ \ \ \ V_{\alpha} = [v^1_{\alpha} , v^2_{\alpha}] ,\ \ \ \ 
\end{equation*}
where $1\leq \alpha, \beta \leq 2$, with each $v^{\beta}_{\alpha}$ being a column vector in ${\mathbb C}^{n-2}$ while each $Y_{\alpha} =(D^{j}_{i\alpha})$ being an $(n-2)\times (n-2)$ matrix. The exact value of the entries of $E_1$ and $E_2$ are determined by the real constants $a,b,c,c',d', \sigma$ defined by:

\begin{equation} \label{abcd}
\left\{ \begin{split} 
[Jx, x] =  ax + by + {\mathfrak a}_J ;\hspace{1.1cm}\\ 
[Jx, y] =  cx -ay + {\mathfrak a}_J ;\hspace{1.1cm} \\ 
\ [Jy, x] =  (c-\sigma )x -ay + {\mathfrak a}_J ; \ \, \\
\ [Jy, y] =  c'x + d'y + {\mathfrak a}_J; \hspace{0.95cm}\\
[Jx,Jy] = \sigma Jx + {\mathfrak a}, \hspace{1.6cm}
\end{split}
\right.
\end{equation}
where $\sigma \neq 0$, $a^2+bc=0$,  and 
\begin{equation} \label{abcd1}
\left\{ \begin{split} \,bc' = -a (c+\sigma ) ; \hspace{0.87cm} \\
\ bd' = 2a^2 +bc - 2b \sigma ;\ \,\\
2ac' + c d' = c^2 . \hspace{0.97cm} 
\end{split} \right.
\end{equation}
Note that the middle equation can be simplified by the condition $a^2+bc=0$, but we keep it in this form so it could be used again in the next case. Using these constants, we know that $J$ is generic $\Longleftrightarrow$ $b\neq 0$; $J$ is half-generic $\Longleftrightarrow$ $b=0$, $c\neq 0$; while $J$ is degenerate $\Longleftrightarrow$ $b=c=0$, thus one has
\begin{equation} \label{abcd3}
\left\{ \begin{split} \mbox{When} \ J \ \mbox{is generic}: \  \  \ b\neq 0, \ c=-\frac{a^2}{b}, \ c'=-\frac{a}{b}(c+\sigma ), \ d'= -c -2\sigma ;   \ \ \ \ \\
\mbox{When} \ J \ \mbox{is half-generic}: \ \   c\neq 0, \ a=b=0, \ d'=c, \ \mbox{and} \ c' \ \mbox{is artbitrary}; \ \ \\
\mbox{When} \ J \ \mbox{is degenerate}:  \ \ \  a=b=c= 0, \ c' \ \mbox{and} \ \, d' \ \mbox{are artbitrary.} \hspace{1.2cm}
\end{split} \right.  
\end{equation}
By formula (16) - (23) and Lemma 7 (note that our $Y_1$ may no longer be nilpotent here) of  \cite{CaoZ}, we have
\begin{eqnarray}
&& C^1_{12} \ = \   -\frac{\ i\sigma }{\sqrt{2}\delta'}, \ \ \ C^2_{12}=0,    \label{C12}\\
&& D^1_{11} \ = \  \frac{b\delta +ia}{\sqrt{2}} , \ \ \ \ D^1_{21} \ = \  \frac{\,ib\delta'}{\sqrt{2}} .  \label{D1} \\
 && D^2_{11} \ = \ \frac{1}{\sqrt{2}\delta'} (-2a\delta +ic+ib\delta^2), \ \ \ D^2_{21} \ = \  - \frac{1}{\sqrt{2}} (b\delta + ia ) , \label{D2}\\
&& D^1_{12} \ = \ \frac{1}{\sqrt{2}\delta'} i(c-\sigma -b\delta^2),  \ \ \ \ \ \ D^1_{22} \ = \ \frac{1}{\sqrt{2}}(b\delta -ia).  \label{D3} \\
&& D^2_{12} \ = \ \frac{1}{\sqrt{2}\delta'^2} \big( i(c'+a\delta^2) + \delta (d'+b\delta^2+ \sigma ) \big) , \ \ \ \ D^2_{22} \ = \   \frac{1}{\sqrt{2}\delta'} i(d'+ b\delta^2 ) . \label{D4}  \\
  && [E_1, E_2] \ = \ \frac{i\sigma}{\sqrt{2}\delta'} E_1, \ \ \ \ \  \ \ \ [Y_1, Y_2] \ = \ \frac{i\sigma}{\sqrt{2}\delta'} Y_1,   \label{Jacobi1}\\
  && V_1E_2+Y_1V_2-V_2E_1 - Y_2V_1 \ = \ \frac{i\sigma}{\sqrt{2}\delta'} V_1, \label{Jacobi2} \\
  && \mbox{tr}(Y_2)  - \overline{\mbox{tr}(Y_2) } + 2t \, \overline{\mbox{tr}(Y_1)} = \ \frac{i}{\sqrt{2}\delta' } \big( 2 \sigma - d' -c \big).  \label{unimo}
\end{eqnarray}

\vspace{0.1cm}

\noindent {\bf Case 2:} $J{\mathfrak a} \neq {\mathfrak a}$ and ${\mathfrak g}/{\mathfrak a}$ is abelian.

\vspace{0.2cm}

Recall that the rank of  $({\mathfrak g},J)$ (or for simplicity, just $J$) is defined as  $r_0=\dim_{\mathbb R}\{({\mathfrak g}'+{\mathfrak a}_J)/{\mathfrak a}_J\}$, where ${\mathfrak g}'=[{\mathfrak g}, {\mathfrak g}]$. It is an integer between $0$ and $2$. So in this case $({\mathfrak g},J)$ are again divided into three mutually disjoint subcases: $r_0=0$, $r_0=1$, or $r_0=2$. 

Take any $V\cong {\mathbb R}^2$ so that ${\mathfrak a}={\mathfrak a}_J \oplus V$, then ${\mathfrak g}={\mathfrak a} \oplus JV$. If $\{x,y\}$ is a basis of $V$, then all the non-trivial Lie bracket info are given by $Jx$ and $Jy$, and $[Jx,Jy]\in {\mathfrak a}_J$ by the integrability of $J$. Replace $\{x,y\}$  by another basis of $V$ if necessary, we may always assume that the matrix $\mbox{ad}_{Jx}|_V$ has zero trace. Under this assumption, one can introduced the real constants $a,b,c,c',d'$ as in the previous case, obeying (\ref{abcd}) and (\ref{abcd1}). 

Now if $g$ is a metric compatible with $J$, then by choosing $V= {\mathfrak a}\cap {\mathfrak a}_J^{\perp}$ and $\{ x, y\}$ an orthonormal basis of $V$ such that $\mbox{ad}_{Jx}|_V$ is trace free, we get an {\em admissible frame} in exactly the same way as in the previous case, and the formula for structure constants (\ref{C12}) through (\ref{unimo}) are all the same, with the only difference being that $\sigma =0$ and we no longer have $a^2+bc=0$.

\vspace{0.3cm}

\section{The proof of Proposition \ref{thm1}}

Let ${\mathfrak g}$ be a unimodular Lie algebra containing a codimension $2$ abelian ideal ${\mathfrak a}$. Let $J$ be a complex structure on ${\mathfrak g}$ so that $J{\mathfrak a} \neq {\mathfrak a}$. The goal of this section is to prove Proposition \ref{thm1}, namely, to show that $({\mathfrak g},J)$ can  not admit any Hermitian-symplectic metric, unless it admits a K\"ahler metric.

Assume that $g$ is a Hermitian-symplectic metric on $({\mathfrak g},J)$. Then under any unitary frame $e$, by Lemma \ref{lemma1} we know that there exists a skew-symmetric matrix $S$ so that the two equations of $(\ref{eq:HS})$ hold. If we choose $e$ to be admissible, then by Lemma \ref{lemma2} the structure constants $C$ and $D$ satisfy all those restrictions, and we have 

\begin{lemma}\label {lemma3}
Let $({\mathfrak g},J,g)$ be a unimodular Lie algebra with Hermitian structure and ${\mathfrak a}\subseteq {\mathfrak g}$ an abelian ideal of codimension $2$ with $J{\mathfrak a} \neq {\mathfrak a}$. Let $e$ be an admissible frame. If $g$ is Hermitian-symplectic, then there exists a constant $p$, two column vectors $u_1$, $u_2\in {\mathbb C}^{n-2}$, and a skew-symmetric $(n-2)\times (n-2)$ matrix $S$ so that
\begin{eqnarray}
&& Y_1-Y_1^{\ast}=0, \ \ \ Y_2 -Y_2^{\ast}+2t\,Y_1^{\ast} =0;  \label{L1}\\
&& v^2_1-v^1_2 -2t \,v^1_1 =0; \label{L2} \\
&& Y_{\alpha}^{\ast}S+ S\overline{Y}_{\alpha} =0,  \ \ \ \ 1\leq \alpha \leq 2;  \label{L3} \\
&& \sum_{\gamma =1}^2 \overline{D^{\beta}_{\gamma \alpha} } \,u_{\gamma} + S \,\overline{v^{\beta}_{\alpha}} + Y^{\ast}_{\alpha}u_{\beta} = \sqrt{-1} \,v^{\alpha}_{\beta} , \ \ \ 1\leq \alpha, \beta \leq 2; \label{L4} \\
&& -p \cdot \overline{\mbox{tr}(E_{\alpha})} + \langle u_2, \overline{v^1_{\alpha}} \rangle - \langle u_1, \overline{v^2_{\alpha}} \rangle = - \sqrt{-1} \,T^{\alpha}_{12}, \ \ \ 1\leq \alpha \leq 2; \label{L5} \\
&&   C^1_{1 2}u_1 + C^2_{12} u_2 - (Y^{\ast}_2 -2t Y^{\ast}_1)u_1 + Y^{\ast}_1 u_2 = 0.   \label{L6}
\end{eqnarray}
\end{lemma}

\begin{proof} Since $g$ is Hermitian-symplectic, by Lemma \ref{lemma1} there will be a skew-symmetric $n\times n$ matrix, which we will denote by $\tilde{S}$, so that the two equations in (\ref{eq:HS}) hold. Write
$$ \tilde{S}  =  \left[ \begin{array}{ccc} 0 & -p & -\,^t\!u_1 \\ p & 0 & -\,^t\!u_2 \\ u_1 & u_2 & S \end{array} \right] , $$ 
where $S$ is the lower right $(n-2)\times (n-2)$ corner.  Since $D^{\ast}_{\ast j}=0$ for any $3\leq j\leq n$, the second equation of (\ref{eq:HS}) gives us 
$$ 0 = T^j_{ab}=-C^j_{ab}-D^j_{ab} + D^j_{ba} =0, \ \ \ \  \forall \ 1\leq a,b\leq n. $$
Let us choose $a=i$ and $b=\alpha$, or $a=1$ and $b=2$, we get by Lemma \ref{lemma2} that
$$ Y_1-Y_1^{\ast} =0, \ \ \ Y_2 -Y_2^{\ast}+2tY_1^{\ast} =0, \ \ \ v^2_1-v^1_2 -2t v^1_1 =0. $$
Here $\ast$ stands for conjugate transpose, and we have used the fact that $\bar{t}=-t$ since $t$ is pure imaginary. This gives (\ref{L1}) and (\ref{L2}). By letting $j=\alpha \in \{ 1,2\}$ and $i, k\in \{ 3, \ldots , n\}$ in the second equation of (\ref{eq:HS}), we get (\ref{L3}). Similarly, by letting $j=\alpha$, $k=\beta \in \{ 1,2\}$ and $i\in \{ 3, \ldots , n\}$ in (\ref{eq:HS}), we get (\ref{L4}). Also, taking $j=\alpha\in \{ 1,2\}$ and $i=1$, $k=2$ in (\ref{eq:HS}), we get (\ref{L5}). For the first equation in (\ref{eq:HS}), note that the three indices $i$, $j$, $k$ need to be distinct, otherwise the identity holds automatically. When $i,j,k$ are all in $\{ 3, \ldots , n\}$, the equation holds trivially. When $i,j\in \{ 3, \ldots , n\}$ and $k=\alpha \in \{ 1,2\}$, we recover (\ref{L3}) again. When $i\geq 3$ but $j=1$, $k=2$, and utilizing (\ref{L2}), we get (\ref{L6}). This completes the proof of the lemma. 
\end{proof}

Now we are ready to prove Proposition \ref{thm1} in the case when ${\mathfrak g}/{\mathfrak a}$ is non-abelian. 

\begin{proof}[{\bf Proof of Proposition \ref{thm1} for the $J{\mathfrak a}\neq {\mathfrak a}$ and ${\mathfrak g}/{\mathfrak a}$ non-abelian case.}] 

Since  $\sigma \neq 0$, the second equation of (\ref{Jacobi1}) implies that $Y_1$ must be nilpotent. But $Y_1$ is also Hermitian symmetric by (\ref{L1}), thus it must be zero: $Y_1=0$. So by (\ref{L1}) again we get $Y_2=Y_2^{\ast}$. Now (\ref{unimo}) yields 
$$ d'=2\sigma-c. $$
Compare this with (\ref{abcd3}), we know that $J$ cannot be generic as $\sigma \neq 0$, while in the other two cases by (\ref{abcd1}) and (\ref{C12}) through (\ref{D4}) we obtain:
\begin{equation} \label{H}
\mbox{(half-generic case):} \ \ \ a=b=0,\ \  c=d'=\sigma , \ \  E_1 = \left[ \begin{array}{cc} 0 & q \\ 0 & 0 \end{array} \right] , \ \ E_2 = \left[ \begin{array}{cc} 0 & q' \\ 0 & q \end{array} \right], 
\end{equation}
where $c'$ is arbtrary, $q=\frac{i\sigma}{\sqrt{2}\delta'}$, $q'=\frac{1}{\sqrt{2}\delta'^2}(ic' + 2\sigma \delta )$, and 
\begin{equation*} \label{D}
\mbox{(degenerate case):} \ \ \ a=b=c=0,\ \  d'=2\sigma , \ \  E_1 = 0 , \ \ E_2 = \left[ \begin{array}{cc} -q & q'' \\ 0 & 2q \end{array} \right], 
\end{equation*}
where $c'$ is again arbtrary and  $q''=\frac{1}{\sqrt{2}\delta'^2}(ic' + 3\sigma \delta )$. In both of these cases we have $D^1_{11}=D^1_{21}=0$, thus by (\ref{L4}) with $\alpha =\beta =1$, we get
$$ S\,\overline{v^1_1} = \sqrt{-1}\,v^1_1 . $$
Multiply by $(v^1_1)^{\ast}$ from the left on both sides, and using the fact that $S$ is skew-symmetric, we get $0=\sqrt{-1}\,|v^1_1|^2$, hence $v^1_1=0$. By (\ref{L2}), we get $v^2_1=v^1_2$, and let us denote it by $\xi$.  

In the degenerate case, we have $E_1=0$. By letting $\alpha=1$, $\beta =2$ in (\ref{L4}) we get $S\overline{\xi} =\sqrt{-1}\, \xi$. Multiplying by $\xi^{\ast}$ from the left we conclude that  $\xi=0$.  Now by the $\alpha =1$ case of (\ref{L5}) we end up with $T^1_{12}=0$, which is a contradiction since
$$ T^1_{12}=-C^1_{12} - D^1_{12} + D^1_{21} = q +q +0 = 2q = \frac{2i\sigma}{\sqrt{2}\delta'} \neq 0. $$
Therefore the degenerate case is ruled out, and we will assume that we are in the half-generic case from now on. In this case, since $Y_1=0$, the equation (\ref{Jacobi2}) gives us 
$$ V_1E_2-V_2E_1 -Y_2V_1 = - C^1_{12}V_1 = qV_1. $$   
But $V_1=(0,\xi )$, $V_2=(\xi, v^2_2)$, and $E_1$, $E_2$ are given by (\ref{H}), so we end up with $Y_2\xi = -q\xi$. Since $q\neq 0$ is pure imaginary, while $Y_2$ is Hermitian-symmetric thus only have real eigenvalues, therefore we must have $\xi =0$. So by the $\alpha =1$ case of (\ref{L5}) we end up with $T^1_{12}=0$, which is again  a contradiction as
$$ T^1_{12}=-C^1_{12} - D^1_{12} + D^1_{21} = q +0 +0 = q = \frac{i\sigma}{\sqrt{2}\delta'} \neq 0. $$
We have thus completed the proof of Proposition \ref{thm1} in the ${\mathfrak g}/{\mathfrak a}$ non-abelian case.
\end{proof}

We are now left with the ${\mathfrak g}/{\mathfrak a}$ abelian case. Assuming that $g$ is Hermitian-symplectic, then by (\ref{L1}) and  (\ref{unimo}) we obtain $ d'=2\sigma -c = -c$, since $\sigma =0$ now. By (\ref{abcd1}), we have
\begin{equation*}
bc'=-ac, \ \ \ \ \ \ a^2+bc=0, \ \ \ \ \ \ ac'=c^2. 
\end{equation*}
We will divide the discussion into the following three subcases.

\vspace{0.15cm}

{\em Subcase 1:} $c'=b=0$.

In this case, all six numbers $a,b,c,c',d', \sigma$ are zero, hence $E_1=E_2=0$ and we are in the $r_0=0$ situation.

\vspace{0.15cm}

{\em Subcase 2:} $c'=0$, $b\neq 0$.

In this case, $b$ is the only non-zero number among the six constants. We have 
\begin{equation*}
E_1 = \kappa \left[ \begin{array}{cc} -t & -t^2 \\ 1 & t  \end{array} \right] , \ \ \ \ E_2 = -t E_1, \ \ \ \  \ \ \ \ \mbox{where} \ \ t=\frac{i\delta}{\delta'}, \ \, \kappa = \frac{ib\delta'}{\sqrt{2}} \neq 0.
\end{equation*}

\vspace{0.15cm}

{\em Subcase 3:}  $c'\neq 0$. 

In this case we have $ a= \frac{c^2}{c'}$, $b=-\frac{c^3}{c'^2}$, and
\begin{equation*}
 E_1= -\frac{ic\delta'}{\overline{\mu}} E_2, \ \ \ \ E_2 =  \frac{\overline{\mu}}{\sqrt{2}c'^2\delta'^2}     \left[ \begin{array}{cc} ic\delta'\mu  & \mu^2 \\ c^2\delta'^2 & -ic\delta' \mu  \end{array} \right] , \ \ \ \ \mu = ic'-c\delta \neq 0.
\end{equation*}
Note that in both of the last two subcases, we have $r_0=1$. The presence of Hermitian-symplectic metric rules out the $r_0=2$ case already. Let us deal with Subcase 2 first. Perform a unitary change of the frame $e$:
$$ \tilde{e}_i= e_i, \ (3\leq i\leq n), \ \ \ \ \left[ \begin{array}{c} \tilde{e}_1 \\ \tilde{e}_2  \end{array} \right] = U  \left[ \begin{array}{c} e_1 \\ e_2  \end{array} \right]  = \delta' \left[ \begin{array}{cc} 1 & t  \\ t & 1  \end{array} \right] \left[ \begin{array}{c} e_1 \\ e_2  \end{array} \right] , $$
Since $  D^{\gamma}_{\alpha \beta} = \overline{\varphi}_{\alpha} ( [ \overline{e}_{\gamma} , e_{\beta}])$, we deduce that
$$ \tilde{E}_{\beta} = (\tilde{D}^{\gamma}_{\alpha \beta} ) = \sum_{\beta_1} U_{\beta \beta_1} (U E_{\beta_1} U^{\ast}) = (U_{\beta 1}\kappa - U_{\beta 2} t\kappa ) \frac{1}{\delta'^2} \left[ \begin{array}{cc} 0 & 0  \\ 1 & 0  \end{array} \right] .$$
Therefore,
$$ \tilde{E}_1 = \frac{\kappa}{\delta'^4} \left[ \begin{array}{cc} 0 & 0  \\ 1 & 0  \end{array} \right]  , \ \ \ \ \ \ \ \tilde{E}_2 =0.  $$
So under this new admissible frame, which for simplicity let us drop the tilde sign, the only non-trivial components for $D^{\gamma}_{\alpha \beta }$ is $D^1_{21}\neq 0$, while we still maintain $C^1_{12}=C^2_{12}=0$. Hence by the structure equation (\ref{structure}) and Lemma \ref{lemma2}, we have
$$ d \varphi_2 = - \overline{D^1_{21}} \, \varphi_1 \wedge \overline{\varphi}_1.  $$
In particular, $i\varphi_1 \wedge \overline{\varphi}_1$ is a non-negative $(1,1)$-form which is $d$-exact and non-trivial, contradicting to the existence of Hermitian-symplectic metrics by Lemma \ref{lemma1b}. Therefore we know that this subcase can not occur. 

Likewise, subcase 3 can be dealt with analogously, by choosing 
$$ U = \frac{1}{\lambda}  \left[ \begin{array}{cc} ic\delta'  & \mu  \\ -\overline{\mu} & -ic\delta'  \end{array} \right],   $$
where $\lambda = \sqrt{ c'^2+c^2}>0$. After the frame change, the only non-zero components of  $D^{\gamma}_{\alpha \beta }$ is $D^1_{21}\neq 0$ once again, 
 and the rest of the argument go through verbatim. So we are only left with the $r_0=0$ case.

\begin{proof}[{\bf Proof of Proposition \ref{thm1} for the $J{\mathfrak a}\neq {\mathfrak a}$ and ${\mathfrak g}/{\mathfrak a}$ abelian case.}] 
Suppose $({\mathfrak g}, J, g)$ is a unimodualr Hermitian Lie algebra with $g$ being Hermitian-symplectic and ${\mathfrak g}$ containing a codimension $2$ ableian ideal ${\mathfrak a}$ such that $J{\mathfrak a} \neq {\mathfrak a}$ and ${\mathfrak g}/{\mathfrak a}$ abelian. From the discussion above, we know that we only need to deal with the $r_0=0$ case, namely, we have $C^1_{12}=C^2_{12}=0$ and $E_1=E_2=0$. Since  $[Y_1, Y_2]=0$, by (\ref{L1}) we know that with an appropriate  unitary change of $\{ e_3, \ldots , e_n\}$ we may assume that 
$$ Y_1 = \Lambda_1, \ \ \ Y_2 = \Lambda_2 -t \Lambda_1, $$
where $\Lambda_1$, $\Lambda_2$ are diagonal, real matrices.

Equations (\ref{L3}) now give us 
$$ \Lambda_{\alpha} S+ S\Lambda_{\alpha} =0, \ \ \ \ \ 1\leq \alpha \leq 2. $$
Since $Y_{\alpha}$ or $Y^{\ast}_{\alpha}$ are linear combinations of $\Lambda_1$ and $\Lambda_2$, we have
$$  Y_{\alpha} S + SY_{\alpha} =0, \ \ \ \ Y^{\ast}_{\alpha} S + SY^{\ast}_{\alpha} =0. $$
Therefore,
\begin{equation*} \label{YScommute}
 Y^{\ast}_{\alpha} (S^{\ast} S)= (SY_{\alpha})^{\ast} S = - (Y_{\alpha}S)^{\ast}S=    - S^{\ast}Y^{\ast}_{\alpha}S = (S^{\ast} S)Y^{\ast}_{\alpha} , \ \ \ \ 1\leq \alpha \leq 2. 
 \end{equation*}
Similarly, $Y_{\alpha} (S^{\ast} S)= (S^{\ast} S) Y_{\alpha}$ for $1\leq \alpha \leq 2$. 
Now the equation (\ref{L4}) becomes
$$ S \overline{ v^{\beta}_{\alpha} } + Y^{\ast}_{\alpha} u_{\beta} = i\,v^{\alpha}_{\beta}, \ \ \ \ 1\leq \alpha, \beta \leq 2. $$ 
Since $S$ is skew-symmetric, by multiplying $S^{\ast}$ from the left onto the above equation, we get
\begin{eqnarray*}
S^{\ast}S \,\overline{ v^{\beta}_{\alpha} } & = &  S^{\ast} (i\,v^{\alpha}_{\beta} - Y^{\ast}_{\alpha} u_{\beta}) \ = \  -i \, \overline{(S \, \overline{v^{\alpha}_{\beta} })} - S^{\ast} Y^{\ast}_{\alpha} u_{\beta} \\
& = & -i \,\overline{( i  \,v^{\beta}_{\alpha}  - Y^{\ast}_{\beta}u_{\alpha}  ) } - S^{\ast} Y^{\ast}_{\alpha} u_{\beta} \\
& = & - \overline{v^{\beta}_{\alpha}} + i \,^t\!Y_{\beta} \,\overline{u}_{\alpha} + Y^{\ast}_{\alpha} S^{\ast}u_{\beta}.
\end{eqnarray*}
Here on the second line we used (\ref{L4}) again. Since $I+S^{\ast}S>0$, the above gives us
\begin{equation} \label{expressv}
  \overline{ v^{\beta}_{\alpha} } =  (I+S^{\ast}S)^{-1} ( i Y_{\beta} \overline{u}_{\alpha} + Y^{\ast}_{\alpha} S^{\ast}u_{\beta})= Y_{\beta}(i(I+S^{\ast}S)^{-1}\overline{u}_{\alpha}) + Y^{\ast}_{\alpha} ( (I+S^{\ast}S)^{-1}S^{\ast}u_{\beta} ).
\end{equation}
For each $\alpha =1,2$, as $Y_{\alpha}$ is diagonal, the image space of $Y_{\alpha}$ and $Y^{\ast}_{\alpha}$ coincide: $\mbox{Im}(Y_{\alpha})= \mbox{Im}(Y^{\ast}_{\alpha})$.  Also, since $t$ is pure imaginary, when $t\neq 0$ we have $\mbox{Im}(Y_1) \subseteq \mbox{Im}(Y_2)$. Equation (\ref{expressv}) tells us that 
$$v^1_1\in \mbox{Im}(Y_1), \ \ \ \ \ v^2_2\in \mbox{Im}(Y_2). $$
Also, by (\ref{L6}), which now reads $(Y^{\ast}_2-2t Y^{\ast}_1)u_1= Y^{\ast}_1u_2$, we get $Y^{\ast}_2u_1 = Y^{\ast}_1(u_2+2tu_1)$. Plug it into (\ref{expressv}) for $\alpha =1$ and $\beta =2$, we get $v^2_1 \in \mbox{Im}(Y_1)$. Now by (\ref{L2}), $ v^1_2 = v^2_1-2tv^1_1 $, so it belongs to $\mbox{Im}(Y_1)$. Note that when $t\neq 0$, $\mbox{Im}(Y_1) \subseteq \mbox{Im}(Y_2)$, so we have 
\begin{equation} \label{vinimageY}
v^{\beta}_{\alpha}\in \mbox{Im}(Y_{\alpha}), \ \ \ \ \forall \ 1\leq \alpha , \beta \leq 2. 
\end{equation}
If $t=0$, then (\ref{L6}) gives us $Y^{\ast}_1 u_2 = Y^{\ast}_2 u_1$, so by plugging it into  (\ref{expressv}) for $\alpha =2$ and $\beta =1$, we get $v^1_2 \in \mbox{Im}(Y_2)$. So (\ref{vinimageY}) always holds. Let us express (\ref{vinimageY}) as
$$ v^{\beta}_{\alpha} = Y_{\alpha} \overline{a}_{\beta}, \ \ \ \ 1\leq \alpha , \beta \leq 2, \ \ a_{\beta} =\,^t\!(a_{\beta3}, \ldots , a_{\beta n})\in {\mathbb C}^{n-2}. $$
Let 
$$\tilde{e}_i = e_i, \ \ (3\leq i\leq n), \ \ \  \tilde{e}_{\alpha} =e_{\alpha} + \sum_{i=3}^n a_{\alpha i} e_i, \ \ (1\leq \alpha \leq 2). $$
Use $\{ \tilde{e}_1, \ldots , \tilde{e}_n\}$ to be a unitary frame, we obtain a new Hermitian metric $\tilde{g}$ on $({\mathfrak g}, J)$, and by formula (41) in \cite{CaoZ} we have
$$ \tilde{Y}_{\alpha} =Y_{\alpha}, \ \ \ \tilde{v}^{\beta}_{\alpha } = v^{\beta}_{\alpha} - Y_{\alpha} \overline{a}_{\beta} = 0. $$
So by \cite[Lemma 16]{CaoZ} we know that $\tilde{g}$ is K\"ahler. This completes the proof of Proposition \ref{thm1}. 
\end{proof}

\vspace{0.3cm}

\vs

\noindent\textbf{Acknowledgments.} The second named author would like to thank Bo Yang and Quanting Zhao for their interests and/or helpful discussions. 

\vs

\end{document}